\documentclass[12pt]{article}
\usepackage{cite}
\usepackage{fullpage}
\usepackage{amssymb}
\usepackage{amsthm}
\usepackage{amsmath}
\usepackage{graphicx}
\usepackage[vcentermath]{youngtab}
\usepackage[all,2cell]{xy}
\usepackage{listings}
\usepackage{young}
\usepackage{enumitem}
\usepackage{youngtab}
\usepackage[colorlinks=true, linkcolor=blue, citecolor=blue, urlcolor=blue]{hyperref}

\newtheorem{theorem}{Theorem}[section]

\newtheorem{example}[theorem]{Example}
\newtheorem{lemma}[theorem]{Lemma}

\newtheorem{proposition}[theorem]{Proposition}
\newtheorem{observation}[theorem]{Observation}

\newcommand{\eat}[1]{}

\newcommand{\grass}[2]{G_{{#1},{#2}}}

\newcommand{\taumn}[2]{w_{#1,#2}}
\newcommand{\gmodp}{G/P}

\newcommand{\gmnmodt}[2]{T \backslash\mkern-6mu\backslash (\grass{#1}{#2})}
\newcommand{\xtaumodt}[2]{T\backslash\mkern-6mu\backslash X(\taumn{#1}{#2})}
\newcommand{\schbmodt}[1]{T\backslash\mkern-6mu\backslash X(#1)}

\begin{document}
\title{Smooth torus quotients of Schubert varieties in the Grassmannian}
\author{
Sarjick Bakshi
\thanks{Chennai Mathematical Institute, Chennai, India, {\tt sarjick@cmi.ac.in}}
\and
S. Senthamarai Kannan \thanks{Chennai Mathematical Institute, Chennai,
  India, {\tt kannan@cmi.ac.in}}
\and
K.Venkata Subrahmanyam \thanks{Chennai Mathematical Institute,
  Chennai, India, {\tt kv@cmi.ac.in}}
}

\date{}
\maketitle

\begin{abstract}

Let $r < n$ be positive integers and further suppose $r$ and $n$ are coprime. We study the GIT quotient of Schubert varieties $X(w)$ in the Grassmannian $\grass{r}{n}$, admitting semistable points for the action of $T$ with respect to the $T$-linearized line bundle ${\cal L}(n\omega_r)$. We give necessary and sufficient combinatorial conditions for the GIT quotient 
$\schbmodt{w}^{ss}_{T}({\cal L}(n\omega_r))$ to be smooth.
\end{abstract}

\section{Introduction}\label{s.Introduction}
Let $G=SL(n,\mathbb{C})$. Let $T$ be the subgroup of all diagonal matrices in $G$, and $B$ the 
subgroup of all upper triangular matrices in $G$ and $B^{-}$ the subgroup of all lower triangular matrices in $G$. Let $X(T)$ denote the group of characters of $T$.  In the root system $R$ of $(G,T)$ let $R^{+}$ denote the set of positive roots with respect to $B$.  Then $S = \{\alpha_1,\ldots, \alpha_{n-1}\}$ is the set of simple roots where $\alpha_i = \epsilon_i - \epsilon_{i+1}$, see \cite[Chapter 3]{lakshmibai2007standard}. Here $\epsilon_i$ is the character sending $t = (t_1,t_2,\ldots,t_n) \in T$ to $t_i \in {\mathbb C}^*$.  Let $\{\omega_1,\ldots,\omega_{n-1}\}$ be the set of fundamental weights.The unipotent subgroup $U$ is the subgroup of $B$ with diagonal entries 1, and $U^{-}$ is the unipotent subgroup of $B^{-}$ with diagonal entries 1. The unipotent group associated to a root $\beta \in R$ is denoted by $U_{\beta}$. 
Let $N_G(T)$ denote the normalizer of $T$ in $G$. The Weyl group $N_G(T)/T$ of $G$ with respect to $T$ is denoted by $W$.

For each simple root $\alpha_i$ we have a morphism $\phi_i: SL(2, \mathbb{C}) \rightarrow G$, with $\phi_i$ sending $M \in SL(2,\mathbb{C})$ to the $n \times n$ matrix having $M$ in rows and columns $i,i+1$, with the other diagonal entries being $1$ and the remaining entries zero. We use the following notation:
\begin{equation}
\label{eqn:notation}
 \dot{s}_{\alpha_i} = \phi_i\begin{pmatrix} 0 & -1 \\ 1  & 0  \end{pmatrix}\ .\\
\end{equation}
It is easily checked that $\dot{s}_{\alpha_i}$ is in $N_G(T)$. We use the notation $s_{\alpha_i}$ for the coset $\dot{s}_{\alpha_i}T$ in $N_G(T)/T$. The Weyl group $W$ is generated by the $s_{\alpha_i}$.

There is an isomorphism between $W$ and $S_n$, the group of permutations on $n$ symbols, with $s_{\alpha_i}$ mapping to the permutation
swapping $i$ with $i+1$. \eat{ So $W$ is identified with the permutation group $S_n$ and}For simplicity we use the notation $s_i$ for $s_{\alpha_i}$. We sometimes use the one line permutation notation  $(w(1), w(2), \ldots, w(n))$ to denote $w \in W$.  
Let $\{e_1,\ldots,e_n\}$ be the standard basis of $\mathbb{C}^n$. Note that for $r \in \{2,\ldots, n\}$, $P_{\hat{\alpha_r}} = \begin{bmatrix}
          * & * \\
	    0_{n-r,r}  &  * 
          \end{bmatrix}$ is the stabilizer of $<e_1,e_2,\ldots,e_r>$ in $G$. $G/P_{\hat{\alpha_r}}$ is the Grassmannian $\grass{r}{n}$, of $r$-dimensional subspaces of
$\mathbb{C}^n$.  $\grass{r}{n}$ is a projective variety and it carries a transitive action of $G$ making it a homogenous $G$-variety. We denote the subgroup of $W$ generated by simple reflections $s_{\alpha}, \alpha \in S\backslash\{\alpha_r\}$ by $W_{P_{\hat{\alpha_r}}}$. Let $W^{P_{\hat{\alpha_r}}} = \{ w \in W | w(\alpha_j) > 0$ for all $j \neq r \}$. In the sequel we fix $r$ and use the notation $P$ for $P_{\hat{\alpha_r}}$, and use the notation $G/P$ and $\grass{r}{n}$ interchangeably. 

It is well known that there is a bijection between the subsets of $S$ and parabolic subgroups of $G$. The parabolic subgroup corresponding to a subset $A \subseteq S$  is $P_{A}=\cap_{\alpha_i \not \in A} P_{\hat{\alpha_r}}$. The Weyl group $W_{P_{A}}$ of $P_{A}$ is the subgroup generated by simple reflections $\alpha \in A$ and there is a canonical bijection between $W^{P_A}$ and  $W/W_{P_A}$.

The Grassmannian $\grass{r}{n}$ comes with the Pl\"{u}cker embedding 
 $\grass{r}{n} \hookrightarrow \mathbb{P}(\bigwedge^r\mathbb{C}^n)$
 sending each $r$-dimensional subspace to its $r$-th exterior wedge
 product (see \cite{fulton1997young}).  The pull back of ${\cal O}(1)$ from the projective space to $\grass{r}{n}$ is the $T$-linearized line bundle on $G/P$ associated to the one dimensional representation of $P$ given by the character $\omega_r$. So we denote this line bundle on $G/P$ by ${\cal L}(\omega_r)$.

Let ${\cal L}$ be a $T$-linearized ample line bundle on $\grass{r}{n}$. A point $p \in \grass{r}{n}$ is said to be semistable with respect to the $T$-linearized line bundle ${\cal L}$ if there is a $T$-invariant section of a positive power of ${\cal L}$ which does not vanish at $p$.  We denote by $({\grass{r}{n}})^{ss}_T({\cal L})$ 
the set of all semistable points with respect to the $T$-linearized line bundle ${\cal L}$. A point in  $({\grass{r}{n}})^{ss}_T({\cal L})$ is said to be stable if its $T$-orbit is closed in $(\grass{r}{n}^{ss})_T(\cal L)$ and its stabilizer in $T$ is finite.  

Let $({\grass{r}{n}})^{s}_T({\cal L}(\omega_r))$ denote the set of all stable points with respect to the $T$-linearized line bundle ${\cal L}(\omega_r)$.
Skorobogatov \cite{skorobogatov1993swinnerton} and, independently, Kannan(see \cite{kannan1998torus} and \cite{kannan1999torus}) showed that when $r$ and $n$ are coprime semistability is the same as stability.

The $T$ fixed points in $G/P$ are $e_w=wP/P$ with $w \in W^{P}$. The $B$-orbit $C_w$ of $e_w$, is called a Bruhat cell and it is an affine space of dimension $l(w)$. The closure of $C_w$ in $\gmodp$ is the Schubert variety $X(w) \subseteq \grass{r}{n}$.  Now $W_{P}= S_r \times S_{n-r}$,  so the minimal length coset representatives of $W^P$ 
can be identified with $\{w \in W | w(1) < w(2) < \ldots < w(r), w(r+1) < w(r+2) < \ldots < w(n)\}$.  Let  $I(r,n) = \{(i_1,i_2,..i_r) | 1 \leq i_1 < i_2  \cdots <i_r \leq n \}$. Then there is a natural identification of  $W^P$ with $I(r,n)$ sending $w$ to $(w(1),w(2),\ldots,w(r))$.   

For a $T$-linearized line bundle ${\cal L}$ on $\grass{r}{n}$ we use the notation $X(w)^{ss}_T({\cal L})$ (respectively,  $X(w)^{s}_T({\cal L})$) to denote the semistable (respectively, stable) points in $X(w)$ with respect to the restriction of ${\cal{L}}$ to $X(w)$.

It is known from the work of Kumar~\cite{kumar2008descent} and also the work of Kannan and Sardar~\cite{kannan2009torusA} that the line bundle ${\cal L}(d \omega_r)$ descends to the the GIT quotient  $T \backslash\mkern-6mu\backslash X(w)^{ss}_T({\cal L}(d \omega_r))$ precisely when $n|d$.

Kannan and Sardar~\cite{kannan2009torusA} showed that there is a unique minimal Schubert Variety $X(w_{r,n})$ in $G_{r,n}$ admitting semistable points with respect to the line bundle $ \mathcal{L}(n\omega_r)$.
We denote by $(a_1,a_2,\ldots,a_r)$ the unique representative of $w_{r,n}$ in $W^P$ (see also Proposition~\ref{wrn}).
\eat{ Let $w_{r,n}=s_{a_1 -1}\ldots s_1  s_{a_2 -1} \ldots s_2 \ldots s_{a_r -1} \ldots s_r$ be a reduced expression for the unique minimal coset representative of $w_{r,n} \in W^P$}

Our paper is motivated by the question of when the GIT quotient $T \backslash\mkern-6mu\backslash X(w)^{ss}_T({\cal L}(n \omega_r))$ is smooth.\eat{Let $\tilde{ \mathcal{L}}(n\omega_r)$ be the descent of  $\mathcal{L}(n\omega_r)$ to the quotient variety $((\gmnmodt{r}{n})_T^{ss}({\cal L}(n\omega_r))$.} An understanding of the GIT quotient in the case $\text{gcd}(r,n) \neq 1$ is difficult since 
stability is different from semistability.  So we assume that $\text{gcd}(r,n)= 1$. Under this assumption Skorobogotov (see \cite{skorobogatov1993swinnerton}) and Kannan(see \cite{kannan2014git}) showed that the quotient variety $\gmnmodt{r}{n}_T^{ss}({\cal L}(n\omega_r))$ is smooth. In  \cite{bakshi2019torus}, it was shown that the GIT quotient $\xtaumodt{r}{n}^{ss}_{T}({\cal L}(n\omega_r))$ is smooth. In this paper we prove the following theorem.

\begin{theorem}
\label{thm:main}Let $w = (b_1,b_2,\ldots, b_r)$ with $b_i \geq a_i$ for all $i$. Let $X(v_1),\ldots, X(v_k)$, be the $k$ components in the singular locus of $X(w)$. Then the following are equivalent
\begin{itemize}
 \item[$(1)$] $\schbmodt{w}^{ss}_{T}({\cal L}(n\omega_r))$ is smooth.
 \item[$(2)$] For all $i$, we have $ v_i \ngeqslant w_{r,n}$.
 \item[$(3)$] Whenever $b_j \geq b_{j-1}+2$ we have $a_j \geq b_{j-1}+1$.
 \end{itemize} 
\end{theorem}

\section{Semistable points in $G_{r,n}$ and the smooth locus of $X(w)$}

The following proposition was first proved  by Kannan and  Sardar, \cite{kannan2009torusA}. A simpler proof was given in \cite{bakshi2019torus}.
\begin{proposition}\cite{bakshi2019torus}
\label{wrn} Let $r$ and $n$ be coprime. There is a unique minimal Schubert Variety $X(w_{r,n})$ in $G_{r,n}$ admitting semistable points with respect to the $T$-linearized line bundle $ \mathcal{L}(n\omega_r)$. 
 As an element of $I(r,n)$, $w_{r,n}  = (a_1,a_2,\ldots,a_r)$ where $a_i$ is the smallest integer such that $a_i \cdot r \geq i \cdot n$. 
\end{proposition}
Let $w=(b_1,\ldots,b_r)$. Clearly $X(w)$ has semistable points with respect to
 the $T$-linearized line bundle $ \mathcal{L}(n\omega_r)$ if and only if $b_i \geq a_i$ for all $i$.

Now if $w=(b_1,\ldots,b_r) \in I(r,n)$, one reduced expression for the Weyl group element in $W^P$ corresponding to $w$ is $(s_{b_1 -1 } \cdots s_1) \ldots (s_{b_i -1 } \cdots s_i)\ldots(s_{b_r -1} \cdots s_r)$ where a bracket is assumed to be empty is if $b_i -1$ is less than $i$. Since $a_i \geq i+1$ for all $1 \leq i \leq r$, we have
\[w_{r,n} = (s_{a_1 -1 } \cdots s_1) (s_{a_2 -1} \cdots s_2) \ldots (s_{a_r -1} \cdots s_r),\]
and no bracket is empty in the above expression.

In order to prove Theorem~\ref{thm:main},  we need to understand the components in the singular locus of $X(w)$. This is well understood. Since we are working with $G$ of Dynkin type $A$, $X(w)$ and its smooth locus $X(w)_{sm}$ can be described in terms of $G$ and the elements of $W^P$, or, one could use the language of Young tableaux and hooks. One can go from one viewpoint to the other easily, \eat{Since we did not find a reference which gives both these viewpoints} we do so in the next section and also describe the transition from one viewpoint to the other. 

\subsection{Smooth locus of Schubert varieties in $G_{r,n}$}
The singular loci of Schubert varieties in miniscule $G/P$ were determined in ~\cite{lakshmibai1990multiplicities}. There is another description of the singular locus of Schubert varieties $X(w)$ in terms of the stabiliser parabolic subgroup of $X(w)$ due to Brion and Polo (see~\cite{brion1999generic}). They proved the following theorem.
 
\begin{theorem}
\label{brionpolo}
Let $w \in W^P$. Let $P_{w} = \{g \in G | gX(w) = X(w)\}$, the stabilizer of $X(w)$ in $G$.
The smooth locus of the Schubert variety $X(w)$ is $X(w)_{sm}=P_{w} wP/P \subseteq X(w) \subseteq \grass{r}{n}$.
\end{theorem}

We recall the following proposition from~\cite{bai1974cohomology}.
\begin{proposition}
\label{lms74}
Let $w=(b_1,b_2,\ldots,b_r) \in W^P$. Define 
$$J'(w) := \{ j \in [1,\ldots,n-1] \big |  \exists m \ \text{with}\  j = b_m, j+1 \not = b_{m+1}\}.$$
Let $J(w) :=\{1,2,\ldots,n-1\}\setminus J'(w)$. Then $P_w = P_{J}$ where $J = \{ \alpha_j | j \in J(w)\}$
\end{proposition}

We need some more notation to describe the work in~\cite{lakshmibai1990multiplicities}. Let $w = (b_1,b_2,\ldots,b_r)$. Associate to $w$ the increasing sequence $\bf{w} = ({ \bf b_1},{\bf b_2},\ldots,{\bf b_r})$ where ${\bf b_i} = b_{i}-i$, so $0 \leq {\bf b_1} \leq {\bf b_2} \leq \ldots \leq {\bf b_r} \leq n-r$. 
Clearly we have a bijective correspondence between $I(r,n)$ and non-decreasing $r$ length sequences in $0 \leq {\bf b_1} \leq {\bf b_2} \leq \ldots \leq {\bf b_r} \leq  n-r$. An increasing sequence gives us a Young diagram, $Y({\bf w})$, in
an $r \times {n-r}$ rectangle with the $i$-th row having ${\bf b_i}$ boxes\footnote{ rows are numbered $1, \ldots, r$ from bottom to top.}.  We call this the Young diagram corresponding to the Schubert variety $X(w)$.

Recall the following Theorem from \cite{lakshmibai1990multiplicities}\footnote{the notation we use is different from theirs, they work with non-increasing sequences}.
\begin{theorem}[Theorem 5.3 \cite{lakshmibai1990multiplicities}] 
\label{LW90}
Let $X(w)$ be a Schubert variety in the Grassmannian. 
Let $\bf{w} = (p_1^{q_1},\ldots, p_k^{q_k}) = (\underbrace{p_1,\ldots p_1}_\text{$q_1$ times}\ldots,\underbrace{p_k,\ldots p_k}_\text{$q_k$ times})$  be the non-zero parts of the increasing sequence ${\bf w}$ with 
$1 \leq  {\bf p_1} < {\bf p_2} \ldots < {\bf p_k} \leq n-r$. The singular locus  $X(w)$ consists of $k-1$ components. The components are given by the Schubert varieties corresponding to the  Young diagrams 
$Y({\bf w_1}), \ldots Y({\bf w_{k-1}}) $,  where the sequences ${\bf w_i}$ are given by
\[ {\bf w_i} = (p_1^{q_1},\ldots, p_{i-1}^{q_{i-1}},(p_i-1)^{q_i+1}, p_{i+1}^{q_{i+1}-1},p_{i+2}^{q_{i+2}},\ldots,p_r^{q_r}), \]
for 	$1 \leq i \leq r-1$ and $1 \leq p_i < p_{i+1}$.
\end{theorem}
An inner corner in a Young diagram is a box that, if it is removed, still gives us the Young diagram of an non-decreasing sequence.  So an easy to remember description of the irreducible components of the singular locus of
$X(w)$ is as follows :- they are the Schubert varieties in correspondence with Young diagram $Y({\bf w_i)}$ obtained from $Y({\bf w})$ by removing the hook from the $i$-th inner box to the $i+1$-st inner box.

 \section{Smoothness of GIT quotients}
 In this section we first give a criterion for the GIT quotient to be smooth. We then prove the main theorem by showing that if the combinatorial conditions in the statement of the main theorem hold this criterion is met. We assume that $r$ and $n$ coprime.
\begin{theorem}\label{smcri}
Let $w\in W^{P_{\hat{\alpha_r}}}$.  $\schbmodt{w}^{ss}_{T}({\cal L}(n\omega_r))$ is smooth if and only if $X(w)^{ss} \subseteq X(w)_{sm}$.
\end{theorem}
\begin{proof}
Assume that $X(w)^{ss} \subseteq X(w)_{sm}$.  Since $gcd(r,n)=1$, it follows from ~\cite[Corollary 2.5]{skorobogatov1993swinnerton} and ~\cite[Theorem 3.3]{kannan1998torus} that $X(w)^{ss} = X(w)^s$. So the stablizer of all semistable points $x \in X(w)^{ss}$ is finite. The proof now follows along the lines described in ~\cite{kannan2014git}. Suppose $x \in B vP/P$ for some $v$.
Let $R^+(v^{-1})$ denote the set of all positive roots made negative by $v^{-1}$. And choose a subset $\beta_1,\dots, \beta_k$ of positive roots in $R^+(v^{-1})$ such that $x = u_{\beta_1}(t_1)\ldots,u_{\beta_k}(t_k)vP/P$ with $u_{\beta_j}(t_j)$ in the root subgroup $U_{\beta_j}$, $t_j \neq 0$ for $j = 1,\ldots, k$. The isotropy group $T_x$ is $\cap_{i=1}^{i=k} \text{ker}(\beta_j)$. Since this is finite, it follows from ~\cite[Example 3.3]{kannan2014git}  that $T_x = Z(G)$, the center of $G$. Working  with the adjoint group we may assume that the stablizer is trivial. So $\schbmodt{w}^{ss}_{T}({\cal L}(n\omega_r))$ is smooth.

For the converse, first note that since we are in the case $\text{gcd}(r,n)=1$ the quotient is a geometric quotient i.e. it is an orbit space. But then restricted to $X(w)^{ss}$ the quotient is a $T$-bundle. So smooth points go to smooth points in the quotient and non-smooth points go to non-smooth points. Since the quotient is smooth it follows that each point $x \in X(w)^{ss}$ is smooth in  $X(w)^{ss}$. Since ${\cal O}_{x, X(w)^{ss}} = {\cal O}_{x,X(w)}$, it follows that $X(w)^{ss} \subseteq X(w)_{sm}$.
\end{proof}

\eat{\begin{theorem}
Suppose $w= s_{b_1}s_{b_1 -1}\ldots s_1 s_{b_2} s_{b_2 -1} \ldots s_2 \ldots s_{a_r} s_{a_r -1} \ldots s_r$ and $X(s_{\alpha}) \subseteq X(w)$, for all simple roots $\alpha$. Suppose further that 
$b_i \geq a_i, \forall i, 1 \leq i \leq r$. Then $X(w)^{ss} \subseteq X(w)_{sm}$ if and only if the following holds - for all $1 \leq i \leq r-1$ if $b_{i+1} \geq b_i + 2$ and
$b_i \geq b_{i-1} +2$ then $a_{i+1} \geq 1 + b_i$. 
\end{theorem}
\begin{proof} 
Assume that $X(w)^{ss} \subseteq X(w)_{sm}$ and for some $i$ it is the case that $b_{i+1}= b_{i+2}$ and $b_i \geq b_{i-1} +2$.  We prove in this case that $a_{i+1} \geq 1 + b_2$. Later we show how the proof can be modified to handle the case $b_{i+1} \geq b_i + 2$. We construct a Schubert variety $X(v) \subseteq X(w)$ for which $X(v)^{ss}_T$ is nonempty if the condition $a_{i+1} \geq 1 + b_2$ does not hold. We show that $X(v)$ is not in smooth locus of $X(w)$ using Theorem ~\ref{brionpolo}, a contradiction. 
\end{proof}

We use a alternative approach 
We fix a few notations : 
}
We prove the main theorem.

\begin{theorem}Let $w = (b_1,b_2,\ldots, b_r)$ with $b_i \geq a_i$ for all $i$. Let Sing $X(w)$ have $k$ components $X(w_1),\ldots, X(w_k)$. Then the following are equivalent
\label{thm:main}
\begin{itemize}
 \item[$(1)$] $\schbmodt{w}^{ss}_{T}({\cal L}(n\omega_r))$ is smooth.
 \item[$(2)$] $\ w_i \ngeqslant w_{r,n} $ for all $i=1,\ldots,k$.
 \item[$(3)$] Whenever $b_j \geq b_{j-1}+2$ we have $a_j \geq b_{j-1}+1$.
 \end{itemize} 
\end{theorem}

\begin{proof} Since $b_i \geq a_i$ for all $i$ we have that $X(w)^{ss}$ is non empty.  From Theorem \ref{smcri}, $\schbmodt{w}^{ss}_{T}({\cal L}(n\omega_r))$ is smooth if and only if  $w_i \ngeqslant w_{r,n}$ for all $i$. Hence the equivalence of $(1)$ and $(2)$. 

We prove the equivalence of $(2), (3)$. The components of the singular locus of $X(w)$ are Schubert varieties $X(w_i)$ in correspondence with diagrams obtained from $Y({\bf w})$ by removing hooks. There is a hook at row $j$ of $Y({\bf w})$  if and only if $b_j \geq b_{j-1}+2$. We denote the Schubert variety obtained from $w$ by removing the hook at row $j$ by $X(w_j)$. Let the word corresponding to it in $I(r,n)$ be $(b_1', b_2',\ldots, b_r')$. \eat{Clearly $b_j' - j = b_{j-1} - (j-1) -1$ and so $b_j' = b_{j-1}$. }
Now $X(w)$ contains $X(w_{r,n})$.\eat{ and $X(w_j)$ does not contain $X(w_{r,n})$.} Let $t$ be the smallest integer less than $j$ such that $b_{k+1}= b_k +1$ for all $t \leq k < j$. By definition of $w_j$ we have 
\[b'_p = \begin{cases} b_p & \text{$1 \leq p \leq t-1$},\\
b_p -1 & \text{$t \leq p \leq j-1$},\\
b_{j-1} &\text{$p = j$},\\
b_p& \text{$j+1 \leq p \leq r.$}
\end{cases}  
\]

If $a_j \geq b_{j-1}+1$ then $X(w_j)$ does not contain $X(w_{r,n})$ since $b'_j = b_{j-1} < a_j$. Since $(b_j \geq b_{j-1}+2) \implies a_j \geq b_{j-1} + 1$, it follows that $3 \implies 2$.  It remains to prove the converse. $X(w)$ contains $X(w_{r,n})$ so $b_p \geq a_p$ for $1 \leq p \leq r$.  So $X(w_j)$ does not contain $X(w_{r,n})$ if and only if $a_p \geq b_p' + 1 = b_p$, for some $t \leq p \leq j-1$ or if $a_j \geq b'_j +1 = b_{j-1}+1$.  Now $b_p = b_{p-1}+1$ for all $t < p \leq j-1$, and $a_{p} \geq a_{p-1} +1$.  It follows that if for some $t \leq p \leq j-1$, $a_p \geq b_p$ then $a_{p+1} \geq a_p+1 \geq b_p + 1 = b_{p+1}$, and so we conclude that 
$a_{j-1} \geq b_{j-1}$. Then $a_j \geq a_{j-1}+1 \geq b_{j-1}+1$, completing the proof of $2 \implies 3$.

\eat{ We show that this can happen if and only if
$a_{j}  > b'_j$, and since $b_j' = b_{j-1}$ we are done.  If $a_j \leq b'_j$ then $a_{j-1} \leq a_j - 1 \leq b'_j - 1 = b_{j-1}'$.  One can easily show that $a_i \leq b_i' $ for all the other $i$ as well,  contradicting the fact that $X(w_j)$ does not contain $X(w_{r,n})$.} 

An alternate proof of the equivalence of $(2), (3)$ is as follows. 
First assume that for all $j$ for which $b_{j} \geq b_{j-1} + 2$ we have $a_j \geq b_{j-1}+1$.
Now let $u\in W^{P}$ be such that $w_{r,n}\leq u \leq w$. Let the
one line notation for $u$ be $(b_{1}', b_{2}', \ldots, b_{r}')$. Then $a_i \leq b_{i}' \leq b_i$ for all $1\leq i \leq r$.
Define $u_i = s_{b'_i } s_{b'_i +1} \ldots s_{b_i -1}$ for $1 \leq i \leq r$.  Clearly 
$u=u_1(s_{b_{1}-1}\cdots
s_{1})u_{2}(s_{b_{2}-1}\cdots s_{2})\cdots u_{r}(s_{b_{r}-1}\cdots
s_{r})$. 

For every $1 < i \leq  r$, the index of the least simple reflection less than $u_{i}$ in the Bruhat order is $s_{b'_{i}}$ and the index of 
the largest simple reflection less than $u_{i-1}$ in the Bruhat order is $s_{b_{i-1}-1}$. 
Take any $1\leq i \leq r$ for which $b_i \geq b_{i-1}+2$. By our hypothesis we have $b_{i}'\geq
b_{i-1}+1$, so $s_{b_{i-1}} \not \leq u_i$ and $u_i \in P_w$ from Proposition~\ref{lms74}. Further $u_{i}$ and $u_{i-1}$ commute. For each $1 < i \leq r$ for which $b_i = b_{i-1}+1$, $u_i \in P_w$ from Proposition~\ref{lms74}. Clearly  $u_1 \in P_w$. So for all $i$, $u_i \in P_w$. It is easy to check that $u=u_ru_{r-1}\ldots u_2 u_1 w P/P$. Therefore, by Theorem 2.2, $X(v)\subset X(w)_{sm}$.

Now assume that $b_{j} \geq b_{j-1} + 2$ but $a_j \leq b_{j-1}$.
iI follows from the definition of  $J'(w)$ in Proposition~\ref{lms74} that $b_{j-1} \in J'(w)$. 
Let $t$ be the smallest integer less than $j$ such that $b_{k+1}= b_k +1$ for all $t \leq k < j$. Then $w$ has a reduced expression of the form
$$w=w" s_{b_t-1} \ldots s_t s_{b_{t+1}-1} \ldots s_{t+1}\ldots s_{b_{j-1} -1} \ldots s_{j-1} s_{b_j -1} \ldots s_j w'.$$ 
Now consider the Weyl group element
$$u=w" s_{b_t-2} \ldots s_t s_{b_{t+1}-2} \ldots s_{t+1}\ldots s_{b_{j-1} -2} \ldots s_{j-1} s_{b_{j-1} -1} \ldots s_j w'.$$ 

Clearly $u \leq w$ and $u$ is obtained from $w$ by left multiplying with the reduced word 
$$  s_{b_{j-1} } s_{b_{j-1} +1 } \ldots  s_{b_j-2} s_{b_j-1}\ldots  s_{b_{t+2}-1} s_{b_{t+1}-1} s_{b_t-1}.$$
In the one line notation $u=(b_1,\ldots,b_{t-1}, b_t-1,\ldots,b_{j-1}-1, b_{j-1}, b_{j+1},\ldots,b_r)$. 
Note that $J'(u) \subseteq J'(w)$ \eat{and the inclusion is proper if $b_t = b_{t-1}+2$}, so $P_w \subseteq P_u$ and therefore $P_w$ stabilises $X(u)$. Since $u < w$, $w$ is not element of $X(u)$. And so $wP/P \notin P_wuP/P$.\eat{ Since $u < w$, it follows that $wP/P \not \in P_w uP/P$.}  Hence  $uP/P \notin P_wwP/P$. Therefore, by 
Theorem~\ref{brionpolo}, $X(u)$ is in the singular locus of $X(w)$.
\eat{The above word is not in $P_w$ since $b_{j-1} \in J'(w)$. It follows from Theorem~\ref{brionpolo} and Proposition~\ref{lms74} that $X(u)$ in not the smooth locus.}However, if $a_j \leq b_{j-1}$, it can be easily seen that $u \geq w_{r,n}$, implying that $X(u)$ contains a semistable point, a contradiction.
\end{proof}

\section{Examples and non-examples}
We illustrate the proof of the main theorem with a simple example.

\begin{example}
Consider the Schubert variety corresponding to $w=(3,5,7,9)$ in $I(4,9)$. The Young diagram associated to $w$ is given by the increasing sequence ${\bf w} = ({\bf 2}, {\bf 3},{\bf 4}, {\bf 5})$.
Fill this diagram starting with $s_i$ at the leftmost box in row $i$, and filling the boxes to the right of this entry in row $i$ with $s_{i+1},s_{i+2},\ldots,$ in order, all the way to the last box in row $i$. We get the filling
\[
\begin{Young}
$s_4$&$s_5$&$s_6$&$s_7$&$s_8$\cr
$s_3$&$s_4$&$s_5$&$s_6$\cr
$s_2$&$s_3$&$s_4$\cr
$s_1$&$s_2$\cr
\end{Young}
\]
Reading the entries in the Young diagram from right to left in each row, and bottom to top yields $s_2 s_1s_4s_3s_2s_6s_5s_4s_3s_8s_7s_6s_5s_4$, the element in $W^P$ corresponding to the Schubert variety $(3,5,7,9)$.
According to Theorem~\ref{LW90} the singular locus of this Schubert variety has three irreducible components given by the sequences $(1,1,4,5)$,$(2,2,2,5)$ and $(2,3,3,3)$. The corresponding Schubert varieties are given by
the tuples $(2,3,7,9)$, $(3,4,5,9)$ and $(3,5,6,7)$, respectively. The Weyl group elements corresponding to these varieties are $s_1s_2s_6s_5s_4s_3s_8s_7s_6s_5s_4$, $s_2s_1s_3s_2s_4s_3s_8s_7s_6s_5s_4$ and 
$s_2 s_1s_4s_3s_2s_5s_4s_3s_6s_5s_4$ respectively. Note that these words can be obtained by removing the hooks occupied by $s_2s_3s_4$, $s_4s_5s_6$ and $s_6s_7s_8$, respectively, and reading the entries left in the resulting Young diagrams from bottom to top, and right to left in each row - exactly as
we did for $w$.

Let us show for example that the Schubert variety corresponding to the Weyl group element $v=s_1s_2s_6s_5s_4s_3s_8s_7s_6s_5s_4$ is not in the smooth locus by showing that it does not satisfy the hypothesis of Theorem~\ref{brionpolo}. The stabilizer of $X(w)$ is the parabolic subgroup corresponding to 
the subset of simple reflections $\{\alpha | s_{\alpha} w \leq w\}$. In this case it can be checked that this is the parabolic subgroup corresponding to $\{\alpha_1,\alpha_2, \alpha_4, \alpha_6, \alpha_8\}$ which is 
$P_{\hat{\alpha_3}} \cap P_{\hat{\alpha_5}} \cap P_{\hat{\alpha_7}}$. However $v$ is obtained from $w$ by multiplying on the left with $s_3 s_4 s_2$. And this element is not in  $P_{\hat{\alpha_3}} \cap P_{\hat{\alpha_5}} \cap P_{\hat{\alpha_7}}$. It can be similarly shown that the other two components are also not in the smooth locus - the Weyl group elements corresponding to them are obtained from $w$ by multiplying on the left with 
$s_4s_6s_5$ and  $s_6 s_8 s_7$ respectively and these elements are clearly not in $P_{\hat{\alpha_3}} \cap P_{\hat{\alpha_5}} \cap P_{\hat{\alpha_7}}$.
 \end{example}

We conclude with examples of Schubert varieties in $\grass{4}{9}$ whose GIT quotients are singular, and examples of Schubert varieties whose GIT quotients are smooth.

\begin{example} We know from \ref{wrn} that $w_{4,9} = (3,5,7,9)$. A reduced expression for the word  $w_{4,9}$ is 
\[ s_{2}s_{1}s_{4}s_{3}s_{2}s_{6}s_{5}s_{4}s_{3}s_{8}s_{7}s_{6}s_{5}s_{4}. \]
The Young diagram $Y(\bf {w_{4,9}})$ corresponding to $w_{4,9}$ is
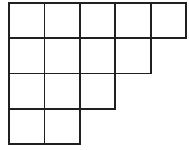
\begin{eqnarray*}
 \yng(5,4,3,2)\\ .
\end{eqnarray*}
Recall from Theorem 3.1 \cite{bakshi2019torus} we have  $\schbmodt{w}^{ss}_{T}({\cal L}(9\omega_4))$ is smooth.
\end{example}
\begin{example}
Let us consider the word $w = (5,7,8,9)$. A reduced expression for $w$ is 
\[ s_{4}s_{3}s_{2}s_{1}s_{6}s_{5}s_{4}s_{3}s_{2}s_{7}s_{6}s_{5}s_{4}s_{3}s_{8}s_{7}s_{6}s_{5}s_{4}. \]
The Young diagram $Y(\bf w)$ is 
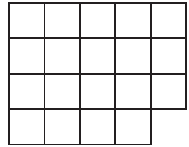
\begin{eqnarray*}
\yng(5,5,5,4)\\ 
\end{eqnarray*}

The singular locus $X(w)$, obtained by removing the only hook corresponds the following tableau:
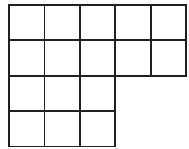
\begin{eqnarray*}
\yng(5,5,3,3)\\ 
\end{eqnarray*}
Here $w^{'} = (4,5,8,9)$. Since $w^{'} > w_{4,9}$,  $X(w^{'})$ contains semistable points and hence the quotient space $\schbmodt{w}^{ss}_{T}({\cal L}(9\omega_4))$ is not smooth (using \ref{thm:main}). 

\end{example}

\begin{example} 
Consider the word $w = (3,5,8,9)$.  A reduced expression for $w$ is 
\[ s_{2}s_{1}s_{4}s_{3}s_{2}s_{7}s_{6}s_{5}s_{4}s_{3}s_{8}s_{7}s_{6}s_{5}s_{4}. \]
The Young diagram $Y(\bf w)$ is 
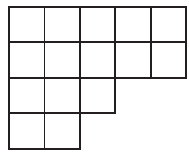
\begin{eqnarray*}
 \yng(5,5,3,2)\\ 
\end{eqnarray*}
The singular locus obtained by removing the hooks has Schubert varieties $X(w_1), X(w_2)$, whose Young diagrams are given by the following tableaux.
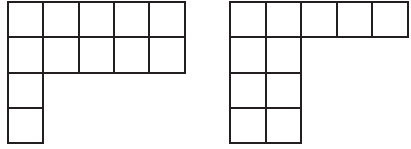
\begin{eqnarray*}
\yng(5,5,1,1) \ \ \ \  \yng(5,2,2,2)\\ 
\end{eqnarray*} 
Here $w_1 = (2,3,8,9)$ and $w_2 = (3,4,5,9)$. Note for $i = 1,2$ $w_i \ngtr w_{4,9}$, so neither $X(w_1)$ nor $X(w_2)$ contain semistable points. Hence the quotient space $\schbmodt{w}^{ss}_{T}({\cal L}(9\omega_4))$ is smooth (using Theorem \ref{thm:main}). 

\end{example}.

\bibliography{references}
\bibliographystyle{plain}

\eat{\appendix
We use the following lemma from Brion and Polo~\cite{brion1999generic}.
\begin{lemma}\label{l.braid}
Fix $1 \leq j < k \leq r$. Let $v = s_{a_j -1} \ldots s_j s_{a_{j+1}-1} \ldots s_{j+1} \ldots s_{a_k -1} \ldots s_k$. Let $j \leq {\ell} \leq a_j -2$. Then 
$s_{\ell} v= v s_t$ for some $t \neq k$, $s_t \leq v$. 

\end{lemma}

We use this to prove the following observation.
\begin{observation}
Fix $1 \leq j < k \leq r$. Let $v = s_{a_j -1} \ldots s_j s_{a_{j+1}-1} \ldots s_{j+1} \ldots s_{a_k-1} \ldots s_k$. Let $j \leq {\ell} \leq a_j -1$.  Then in $W^{P_{\hat{{\alpha_k}}}}$ we have $s_{\ell} v \equiv v\  \text{mod}\  P_{\hat{{\alpha_k}}}$.
\end{observation}
\begin{proof}
The proof is by induction on $k-j$. Write $v = s_{a_j -1} \ldots s_j v'$. Using $s_{\ell} s_{\ell+1}s_{\ell} = s_{\ell+1} s_{\ell} s_{\ell+1}$,  we have $s_{\ell}v = s_{a_j -1} \ldots s_j s_{\ell+1} v'$. Now $v' = s_{a_{j+1}-1} \ldots s_{j+1} \ldots s_{a_k-1} \ldots s_k$. Since $a_j < a_{j+1}$ we have
$j+1 \leq \ell +1 \leq a_j < a_{j+1}$. So by induction $s_{\ell+1} v' \equiv v' \ \text{mod} \ P_{\hat{{\alpha_k}}}$ and so $s_{\ell}v \equiv  s_{a_j -1} \ldots s_j v' \ \text{mod} \ P_{\hat{{\alpha_k}}} \equiv  v \ \text{mod} \ P_{\hat{{\alpha_k}}}$, as required.
\end{proof} 
 
 First, a simple observation.
}
\end{document}